\documentclass{amsart}
\usepackage{amsfonts}
\usepackage{amsmath,amscd,latexsym, amssymb}
\newtheorem{theorem}{Theorem}
\theoremstyle{plain}

\newtheorem{corollary}{Corollary}

\newtheorem{definition}{Definition}
\newtheorem{example}{Example}

\newtheorem{lemma}{Lemma}

\newtheorem{remark}{Remark}

\numberwithin{equation}{section}

\begin{document}
\title {On 1-absorbing primary ideal of a commutative ring (Correction to Theorem 17 is added)}
\author{Ayman Badawi}
\address{Department of Mathematics \&Statistics, The American University of
Sharjah, P.O. Box 26666, Sharjah, United Arab Emirates.}
\email{abadawi@aus.edu}
\author{Ece Yetkin Celikel}
\address{Department of Mathematics, Faculty of Science and Art, Gaziantep
University, Gaziantep, Turkey}
\email{yetkinece@gmail.com}


\begin{abstract}
Let $R$ be a commutative ring with nonzero identity. In this paper, we
introduce the concept of 1-absorbing primary ideals in commutative rings. A
proper ideal $I$ of $R$ is called a {\it $1$-absorbing primary} ideal of $R$ if whenever nonunit elements
$a,b,c\in R$ and $abc\in I$, then $ab\in I$ or $c\in \sqrt{I}.$ Some
properties of 1-absorbing primary ideals are investigated. For example, we show that if $R$ admits a 1-absorbing primary ideal that is not a primary ideal, then $R$ is a quasilocal ring. We give an example of a 1-absorbing primary ideal of $R$ that is not a primary ideal of $R$. We show that if a ring $R$ is not a quasilocal, then a proper ideal $I$ of $R$ is a 1-absorbing primary ideal of $R$ if and only if $I$ is a primary ideal. We show that if $R$ is a Noetherian domain, then $R$ is a Dedekind domain if and only if every nonzero proper 1-absorbing primary ideal of $R$ is of the form $P^n$ for some nonzero prime ideal $P$ of $R$ and a positive integer $n \geq 1$. We show that a proper ideal $I$ of $R$ is a 1-absorbing primary ideal of $R$ if and only if  whenever $I_{1}I_{2}I_{3}\subseteq I$ for some proper ideals $I_{1},I_{2},I_{3}$ of $R$, then $I_{1}I_{2}\subseteq I$ or $I_{3}\subseteq \sqrt{I}.$
\end{abstract}

\maketitle

\section{Introduction}

Throughout this paper all rings are commutative with $1\neq 0.$ Let $R$ be a
commutative ring. If $R$ has exactly one maximal ideal, then $R$ is called a {\it quasilocal} ring. An ideal $I$ of $R$ is said to be proper if $I\neq R$. Let $I$ be a proper ideal of a commutative ring $R$. Then
the radical of $I$ is denoted by $\sqrt{I}=\{r\in R \mid r^{n}\in I$ for some positive integer $n \geq 1 \}$ and the set of zero divisor elements with respect to $I$ is denoted by $Z_{I}(R)=\{r\in R \mid $ $rs\in I$ for some $s\in R\backslash I\}$.

Since prime ideals have an important role in the theory of commutative
rings, there are several ways to generalize the concept of prime ideals.
Badawi generalized the concept of prime ideals in \cite{Badawi}. We recall from \cite{Badawi} that
a nonzero proper ideal $I$ of $R$ is said to be a {\it 2-absorbing}  ideal of $R$
if whenever $a,b,c\in R$ and $abc\in I$, then either $ab\in I$ or $ac\in I$
or $bc\in I$. Anderson and Badawi \cite{AndBad} generalized the notion of
2-absorbing ideals to $n$-absorbing ideals. A proper ideal $I$ of $R$ is
called $n$-absorbing ideal if whenever $x_{1}\cdot \cdot \cdot x_{n+1}\in I$
for $x_{1},...,x_{n+1}\in R$, then there are $n$ of the $x_{i}$'s whose
product is in $I$.  Recall from \cite{Badawi2}
that a proper ideal $I$ of $R$ is called a {\it 2-absorbing primary} ideal of $R$ if
whenever $a,b,c\in R$ with $abc\in I$, then $ab\in I$ or $ac\in \sqrt{I}$ or
$bc\in \sqrt{I}$.

In this paper, we introduce the concept of $1$-absorbing primary ideals of commutative rings . A
proper ideal $I$ of a commutative ring $R$ is called a {\it $1$-absorbing primary} ideal of $R$ if whenever nonunit elements
$a,b,c\in R$ and $abc\in I$, then $ab\in I$ or $c\in \sqrt{I}$. We show that the following implications hold and none of them is revisable:

primary ideal $\Longrightarrow$ 1-absorbing primary ideal $\Longrightarrow$  2-absorbing primary ideal.

Among many results in this paper. We give an example (Example \ref{e1}) of a 1-absorbing primary ideal of $R$ that is not a primary ideal of $R$, and another example (Example \ref{e2}) of a 2-absorbing primary ideal of $R$ that is not a 1-absorbing primary ideal of $R$. We show (Theorem \ref{T-1}) that if $I$ is a 1-absorbing primary ideal of $R$, then $\sqrt{I}$ is a prime ideal of $R$. We show (Theorem \ref{T0}) if a ring $R$ admits a 1-absorbing primary ideal of $R$ that is not a primary ideal, then $R$ is a quasilocal ring. We give a method (Theorem \ref{T2}) to construct 1-absorbing primary ideals of commutative rings that are not primary ideals. We show (Theorem \ref{T3}) that if a ring $R$ is not a quasilocal, then a proper ideal $I$ of $R$ is a 1-absorbing primary ideal of $R$ if and only if $I$ is a primary ideal. We show (Theorem \ref{T12}) that if $R$ is a Noetherian domain, then $R$ is a Dedekind domain if and only if every nonzero proper 1-absorbing primary ideal of $R$ is of the form $P^n$ for some nonzero prime ideal $P$ of $R$ and a positive integer $n \geq 1$. We show (Theorem \ref{T17}) that a proper ideal $I$ of $R$ is a 1-absorbing primary ideal of $R$ if and only if  whenever $I_{1}I_{2}I_{3}\subseteq I$ for some proper ideals $I_{1},I_{2},I_{3}$ of $R$, then $I_{1}I_{2}\subseteq I$ or $I_{3}\subseteq \sqrt{I}.$

For any undefined terminology see \cite{Gilmer}, \cite{Huc}, \cite{Kap}, and \cite{LM}.

\section{Properties of 1-absorbing primary ideals}
We  remind the reader with the following definitions.
\begin{definition}
Let $I$ be a proper ideal of a commutative ring $R$.
\begin{enumerate}
\item We call $I$ a {\it $1$-absorbing primary} ideal of $R$ if whenever nonunit elements
$a,b,c\in R$ and $abc\in I$, then $ab\in I$ or $c\in \sqrt{I}$.
\item (\cite{Badawi2}) We call $I$ a {\it 2-absorbing primary} ideal of $R$ if
whenever $a,b,c\in R$ with $abc\in I$, then $ab\in I$ or $ac\in \sqrt{I}$ or
$bc\in \sqrt{I}$.
\end{enumerate}
\end{definition}

We start with the following trivial result, and hence we omit its proof.

\begin{theorem}
Let $I$ be a proper ideal of $R$. Then
\begin{enumerate}
\item If $I$ is a primary ideal of $R$, then $I$ is a 1-absorbing primary ideal of $R$.
\item If $I$ is a 1-absorbing primary ideal of $R$, then $I$ is a 2-absorbing primary ideal of $R$.
\end{enumerate}
\end{theorem}

The following is an example of a 1-absorbing primary ideal that is not a primary ideal.

\begin{example}
\label{e1}Let $A=K[x,y]$, where $K$ is a field, $M = (x, y)A$, and $R = A_M$. Note that $R$ is a quasilocal ring with maximal ideal $M_M$.  Then $I = xM_M =  (x^{2}, xy)R$ is a
1-absorbing primary ideal of $R$ (see Theorem \ref{T2}) and $\sqrt{I} = xR$. However $xy \in I$, but neither $x \in I$ nor $y \in \sqrt{I}$.  Thus $I$ is not a primary ideal of $R$.
\end{example}

The following is an example of a 2-absorbing primary ideal that is not a 1-absorbing primary ideal.

\begin{example}\label{e2}
 Let $R = \mathbb{Z}$. Consider the ideal $I = 12R$. Then $I$ is a 2-absorbing primary ideal of $R$  by Corollary 2.12 in \cite{Badawi2}. However $2\cdot 2\cdot 3\in I$, but neither $2\cdot 2\in I$ nor  $3\in
\sqrt{I}$. Thus $I$ is not a 1-absorbing primary ideal of $R$.
\end{example}

\begin{theorem}\label{T-1}
Let $I$ be a 1-absorbing primary ideal of a ring $R$. Then $\sqrt{I}$ is a prime ideal of $R$.
\end{theorem}
\begin{proof}
Let $xy \in \sqrt{I}$ for some $x, y \in R$. We may assume that $x, y$ are nonunit elements of $R$. Let $n \geq 2$ be an even positive integer such that $(xy)^n \in I$. Then $n = 2m$ for some positive integer $m \geq 1$. Since $(xy)^n = x^ny^n = x^mx^my^n \in I$ and $I$ is a 1-absorbing primary ideal of $R$, we conclude that $x^mx^m = x^n \in I$ or $y^n \in I$. Hence $x \in \sqrt{I}$ or $y \in \sqrt{I}$. Thus $\sqrt{I}$ is a prime ideal of $R$.
\end{proof}

The following lemma is needed in the proof of our next result.

\begin{lemma}\label{L0}
Let $R$ be a ring. Suppose that for every nonunit element $w$ of $R$ and for every unit element $u$ of $R$, we have $ w + u$ is a unit element of $R$. Then $R$ is a quasilocal ring.
\end{lemma}
\begin{proof}
Suppose that $R$ has at least two maximal ideals, say $M_1, M_2$. Then $m_1 + m_2 = 1$ for some $m_1\in M_1$ and $m_2 \in M$. Thus $1 - m_1 = m_2$  is a unit element of
$R$, which is impossible since $m_2 \in M_2$. Thus $R$ is a quasilocal ring.
\end{proof}

In the following result, we show that if a ring $R$ admits a 1-absorbing primary ideal that is not a primary ideal, then $R$ is a quasilocal ring.

\begin{theorem}\label{T0}
Suppose that a ring $R$ admits a 1-absorbing primary ideal that is not a primary ideal. Then $R$ is a quasilocal ring.
\end{theorem}
\begin{proof}
Suppose that $I$ is a 1-absorbing primary ideal of $R$ that is not a primary ideal of $R$. Hence there exist nonunit elements $x, y \in R$ such that neither $x \in I$ nor $y \in \sqrt{I}$. Let $w$ be a nonunit element of $R$. Since $wxy \in I$ and $I$ is a 1-absorbing primary ideal of $R$ and $y \notin \sqrt{I}$, we conclude that $wx \in I$. Let $u$ be a unit element of $R$. Suppose that $w + u$ is a nonunit element of $R$. Since $(w + u)xy \in I$ and $I$ is a 1-absorbing primary ideal of $R$ and $y \notin \sqrt{I}$, we conclude that $(w + u)x = wx + ux \in I$. Since $wx \in I$, we conclude that $x\in I$, which is a contradiction. Thus $w + u$ is a unit element of $R$. Since for every nonunit element $w$ of $R$ and for every unit element $u$ of $R$, we have $ w + u$ is a unit element of $R$, we conclude that $R$ is a quasilocal ring by Lemma \ref{L0}.
\end{proof}

\begin{theorem} \label{T1}
Suppose that a ring $R$ is not a quasilocal ring. Then a proper ideal $I$ of $R$ is a 1-absorbing primary ideal of $R$ if and only if $I$ is a primary ideal of $R$. In particular, if $R = R_1 \times R_2$ for some rings $R_1$ and $R_2$, then a proper ideal $I$ of $R$ is a 1-absorbing primary ideal of $R$ if and only if $I$ is a primary ideal of $R$.
\end{theorem}
\begin{proof}
If $I$ is a primary ideal of $R$, then it is clear that $I$ is a 1-absorbing primary ideal of $R$. Hence assume that $I$ is a 1-absorbing primary ideal of $R$. Let $xy \in I$ for some $x, y \in R$. We may assume that $x, y$ are nonunit elements of $R$. Suppose that $y \not\in \sqrt{I}$. Since $R$ is not a quasilocal ring, there exist a nonunit element $w \in R$ and a unit element $u \in R$ such that $w + u$ is a nonunit element of $R$ by Lemma \ref{L0}. Since $wxy \in I$ and $I$ is a 1-absorbing primary ideal of $R$ and $y \notin \sqrt{I}$, we conclude that $wx \in I$. Also, since $(w + u)xy \in I$ and $I$ is a 1-absorbing primary ideal of $R$ and $y \notin \sqrt{I}$, we conclude that $(w + u)x = wx + ux \in I$. Since $wx \in I$, we conclude that $x\in I$. Thus $I$ is a primary ideal of $R$.
\end{proof}

Let $R = R_1\times R_2$, where $R_1$ and $R_2$ are commutative rings with $1\not = 0$, and let $J$ be a proper ideal of $R$. Then it is well-know that $J$ is a primary ideal of $R$ if and only if $J = I\times R_2$ for some primary ideal $I$ of $R_1$ or $J = R_1\times L$ for some primary ideal $L$ of $R_2$. Hence, in view of Theorem \ref{T1}, we have the following result.

\begin{theorem}\label{T1.5}
Let $R = R_1\times R_2$, where $R_1$ and $R_2$ are commutative rings with $1\not = 0$, and let $J$ be a proper ideal of $R$. The following statements are equivalent.
\begin{enumerate}
\item $J$ is a 1-absorbing primary ideal of $R$.
\item $J$ is a primary ideal of $R$.
\item $J = I\times R_2$ for some primary ideal $I$ of $R_1$ or $J = R_1\times L$ for some primary ideal $L$ of $R_2$.
\end{enumerate}
\end{theorem}

Recall that a nonzero nonunit element $x$ of a ring $R$ is called {\it irreducible} if $x = x_1x_2$ for some $x_1, x_2 \in R$, then $x_1$ is a unit of $R$ or $x_2$ is a unit of $R$. Also, recall that $x$ is called {\it prime} if $x\mid x_1x_2$ for some $x_1, x_2 \in R$, then $x \mid x_1$ or $x\mid x_2$. The following lemma is needed in the proof of our next result.

\begin{lemma}\label{L2}
Let $R$ be a quasilocal ring. If $p$ is a nonzero prime element of $R$, then $p$ is an irreducible element of $R$.
\end{lemma}
\begin{proof}
Assume that $p = p_1p_2$ for some $p_1, p_2 \in R$. Since $p$ is a prime element of $R$, we may assume that $p\mid p_1$. We show that $p_2$ is a unit of $R$. Hence $p_1 = pw$ for some $w\in R$ and thus $p = pwp_2$. Thus $p - pwp_2 = p(1 - wp_2) = 0$. If $p_2$ is a nonunit element of $R$, then $1 - wp_2$ is a unit of $R$ (since $R$ is quasilocal) and thus $p = 0$, a contradiction. Hence $p_2$ is a unit of $R$. Thus $p$ is an irreducible element of $R$.
\end{proof}

The following result provides a method to construct  1-absorbing primary ideals that are not primary ideals.

\begin{theorem}\label{T2}
Let $R$ be a quasilocal ring with maximal ideal $M$. Let $x \in M$ be a nonzero prime element of $R$ such that $M \not = xR$. Then $xM$ is a 1-absorbing primary ideal of $R$ that is not a primary ideal of $R$.
\end{theorem}
\begin{proof}
Fist, we show that $xM$ is not a primary ideal of $R$. Observe that $\sqrt{xM} = xR$. Since $M \not = xR$, there exists an $m \in M\setminus xR$. Now $xm \in I$. Since $x$ is a nonzero prime element of $R$, we conclude that $x$ is an irreducible element of $R$ by Lemma \ref{L2}, and thus $x\notin xM$. Also, since $m\in M\setminus xR$ and $\sqrt{xM} = xR$, we conclude that $m\notin \sqrt{xM}$. Thus $xM$ is not a primary ideal of $R$. Now we show that $xM$ is a 1-absorbing primary ideal of $R$. Suppose that $abc \in I$ for some nonunit elements $a, b, c \in R$. Suppose that $ab \notin xM$. Then $x\nmid a$ and $x\nmid b$ (note that if $x\mid a$ or $x\mid b$, then $ab \in xM$). Since $x\mid abc$ and $x\nmid ab$, we conclude that $x\mid c$. Thus $c\in \sqrt{xM} = xR$.
\end{proof}

\begin{theorem} \label{T3}
Suppose that $I$ is a 1-absorbing primary ideal of $R$ that is not a primary ideal of $R$. Then there exist an irreducible element $x \in R$ and a nonunit element $y\in R$ such that $xy\in I$, but neither $x\in I$ nor $y\in \sqrt{I}$. Furthermore, if $ab \in I$ for some nonunit elements $a, b \in R$ such that neither $a\in I$ nor $b\in \sqrt{I}$, then $a$ is an irreducible element of $R$.
\end{theorem}
\begin{proof}
Since $I$ is not a primary ideal of $R$, there exist nonunit elements $x, y \in R$ such that neither $x \in R$ nor $y \in \sqrt{I}$.  Suppose that $x$ is not an irreducible element of $R$. Then $x = cd$ for some nonunit elements $c, d \in R$. Since $xy = cdy \in I$ and $I$ is a 1-absorbing primary ideal of $R$ and $y \notin \sqrt{I}$, we conclude that $cd = x \in I$, a contradiction. Hence $x$ is an irreducible element of $R$.
\end{proof}

\begin{theorem}\label{T4}
Let $R$ be a quasilocal ring with maximal ideal $M$ and $P$ be a prime ideal of $R$ such that $P\subseteq M$. Then $PM$ is a 1-absorbing primary ideal of $R$.
\end{theorem}
\begin{proof}.
First observe that $\sqrt{PM} = P$. Suppose that $abc \in PM$ for some nonunit elements $a, b, c \in M$. If $a \in P$ or $b \in P$, then it is clear that $ab \in PM$. Hence assume that neither $a\in P$ nor $b \in P$. Thus $ab \not\in P$. Since  $abc \in PM \subseteq P$ and $ab \notin P$, we conclude that $c \in P = \sqrt{PM}$. Thus $PM$ is a 1-absorbing primary ideal of $R$.
\end{proof}

\begin{remark}
Observe that $PM$ in Theorem \ref{T4} needs not be a primary ideal of $R$ by Theorem \ref{T2}.
\end{remark}

\begin{theorem}\label{T5} Let $I$ be a 1-absorbing primary ideal of a ring $R$. Suppose that $c\in R \setminus I$ is a nonunit element of $R$. Then $(I : c) = \{x \in R \mid cx \in I\}$ is a primary ideal of $R$.
\end{theorem}
\begin{proof}
Suppose that $I$ is a 1-absorbing primary ideal of $R$ and $c$ is a nonunit element of $R$ such that $c \in R\setminus I$. Let $ab\in (I : c)$
for some elements $a, b\in R$. We may assume that $a, b$ are nonunit elements of $R$.  Suppose that $a\notin (I : c)$. Hence $ca \notin I$.  Since $cab \in I$ and $I$ is a 1-absorbing primary ideal of $R$ and $ca \notin I$, we conclude that $b\in \sqrt{I} \subseteq \sqrt{(I : c)}$. Hence $(I : c)$ is a primary ideal of $R$.
\end{proof}

\begin{remark}
Let $I$ be a 1-absorbing primary ideal of a ring $R$ and $c$ be a nonunit element of $R$ such that $c \in R\setminus I$. Suppose that $c \notin \sqrt{I}$. Since $\sqrt{I}$ is a a prime ideal of $R$, we conclude that $I \subseteq (I : c) \subseteq \sqrt{I}$, and thus $\sqrt{(I : c)} = \sqrt{I}$. Suppose that $c\in \sqrt{I}\setminus I$. Let $n$ be the last positive integer $n\geq 2$ such that $c^n \in I$. Then $c^{n-1} \in (I : c) \setminus I$. Thus $I \subsetneq (I : c)$. Also, $(I : c)$ needs not be a subset of $\sqrt{I}$; for let $R$, $I$, $M_M$ be as in Example \ref{e1}, and $c = x$. Then $(I : c) = M_M \nsubseteq \sqrt{I}$.
\end{remark}

Recall that a ring $R$ is called {\it divided} if for every prime ideal $P$ of $R$ and for every $x\in R\setminus P$, we have $x\mid p$ for every $p\in P$. Recall that a ring $R$ is called a {\it chained} ring if for every $x, y \in R$, we have $x\mid y$ or $y\mid x$. Thus every chained ring is divided. Hence if $R$ is a divided ring, then $R$ is a quasilocal ring. We have the following result.

\begin{theorem}\label{T6}
Let $R$ be a divided ring. Then a proper ideal $I$ of $R$ is a 1-absorbing primary ideal of $R$ if and only if $I$ is a primary ideal of $R$. In particular, if $R$ is a chained ring, then a proper ideal $I$ of $R$ is a 1-absorbing primary ideal of $R$ if and only if $I$ is a primary ideal of $R$.
\end{theorem}
\begin{proof}
It is clear that every primary ideal of $R$ is a 1-absorbing primary ideal of $R$. Hence assume that $I$ is a 1-absorbing primary ideal of $R$. Suppose that $xy\in I$ for some $x, y \in R$ and $y\notin \sqrt{I}$. We may assume that $x, y$ are nonunit elements of $R$. Since $\sqrt{I}$ is a prime ideal of $R$ by Theorem \ref{T-1} and $y\notin \sqrt{I}$, we conclude that $x\in \sqrt{I}$. Since $R$ is divided, we conclude that $y \mid x$. Thus $x = yw$ for some $w\in R$. Since $y\notin \sqrt{I}$ and $x \in \sqrt{I}$, we conclude that $w$ is a nonunit element of $R$. Since $xy = ywy \in I$ and $I$ is a 1-absorbing primary ideal of $R$ and $y\notin \sqrt{I}$, we conclude that $yw = x \in I$. Thus $I$ is a primary ideal of $R$.
\end{proof}

Recall that a proper ideal $I$ of $R$ is called {\it principal} if $I = xR$ for some $x\in R$.

\begin{theorem}\label{T7}
Let $R$ be a divided ring with maximal ideal $M$. If $M$ is not a principal prime ideal of $R$, then every nonzero prime ideal of $R$ is not principal.
\end{theorem}
\begin{proof}
 Suppose that a nonzero prime ideal $P = xR$ for some $x\in R$. Then $x$ is a nonzero prime element of $R$ and $M \not = xR$. Thus $xM$ is a 1-absorbing primary ideal of $R$ that is not a primary ideal of $R$ by Theorem \ref{T2}, which is a contradiction by Theorem \ref{T6}.
 \end{proof}

\begin{theorem}\label{T8}
Let $R$ be a divided integral domain and $P$ be a prime ideal of $R$. Then $P^n$ is a primary ideal of $R$ for every positive integer $n\geq 1$, and hence $P^n$ is a 1-absorbing primary ideal of $R$ for every positive integer $n\geq 1$.
\end{theorem}
\begin{proof}
If $n = 1$, then there is nothing to prove. Thus let $n \geq 2$ and Suppose that $xy \in P^n$ for some $x, y \in R$. Then $xy = p_1c_1 + \cdots + p_kc_k \in P^n$ for some $p_1, ..., p_k \in P$ and $c_1, ..., c_k \in P^{n-1}$ for some positive integer $k\geq 1$. Suppose that $y\not\in P$. Since $R$ is divided, we have Then $y\mid p_i$ for every $i$, $1 \leq i \leq k$. Hence for every $i$, $1 \leq i \leq k$, we have $p_i = yd_i$ for some $d_i \in P$. Thus $xy = yd_1c_1 + \cdots + yd_kc_k$. Hence $y(x - (d_1c_1 + \cdots + d_kc_k)) = 0$. Since $R$ is an integral domain, we conclude that $x = d_1c_1 + \cdots + d_kc_k \in P^n$. Thus $P^n$ is a primary ideal of $R$. Since every primary ideal of $R$ is a 1-absorbing primary ideal, we conclude that $P^n$ is a 1-absorbing primary ideal of $R$ for every positive integer $n\geq 2$.
\end{proof}

Recall that an integral domain $R$ is called a {\it valuation} domain if $R$ is a chained ring.

\begin{theorem}\label{T9}
Let $R$ be a valuation domain and $I$ be a proper ideal of $R$ with $\sqrt{I} = P$ (note that $P$ is a prime ideal of $R$). The following statements are equivalents.
\begin{enumerate}
\item $I$ is a 1-absorbing primary ideal of $R$.
\item $I$ is a primary ideal of $R$.
\item If $P \not = P^2$, then $I = P^n$ for some positive integer $n\geq 1$.
\end{enumerate}
\end{theorem}
\begin{proof}
{\bf $(1) \Rightarrow (2)$}. Since $R$ is divided, the claim is clear by Theorem \ref{T6}.

{\bf $(2) \Rightarrow (3)$}. The claim is clear by \cite[Theorem 5.11]{LM}.

{\bf $ (3) \Rightarrow (1)$}. Since $R$ is divided, the claim is clear by Theorem \ref{T8}.
\end{proof}

Let $R$ be an integral domain with quotient field $K$. Recall that a proper ideal $I$ of $R$ is called {\it invertible} if $II^{-1} = R$, where $I^{-1} = \{r \in K \mid rI \subseteq R\}$. An integral domain $R$ is called a {\it Prufer} domain if every nonzero finitely generated ideal of $R$ is invertible.

\begin{theorem}\label{T10}
Let $R$ be a Prufer domain and $I$ be a proper ideal of $R$ with $\sqrt{I} = P$ for some prime ideal $P$ of $R$. Then the following statements are equivalents.
\begin{enumerate}
\item $I$ is a 1-absorbing primary ideal of $R$.
\item $I$ is a primary ideal of $R$.
\item If $P$ is a finitely generated ideal of $R$, then $I = P^n$ for some positive integer $n\geq 1$.
\end{enumerate}
\end{theorem}
\begin{proof}
 {\bf $(1) \Rightarrow (2)$}. Suppose that $R$ is quasilocal with maximal ideal $M$.  Since $R$ is a Prufer domain, it is known that $R = R_M$ is a valuation domain and hence the claim follows from Theorem \ref{T10}. Suppose that $R$ is not quasilocal. Then the  claim follows by  Theorem \ref{T1}.

{\bf $(2) \Rightarrow (3)$}. The claim is clear by \cite[Exercise 2, p. 144]{LM}.

{\bf $ (3) \Rightarrow (1)$}. Since $I = P^n$ for some positive integer $n\geq 1$ and $P$ is a finitely generated ideal of $R$, we conclude that $I$ is a primary ideal of $R$ by \cite[Exercise 2, p. 144]{LM}.  Thus $I$ is a 1-absorbing primary ideal of $R$.
\end{proof}

Recall that an integral domain is called  a {\it Dedekind} domain if every nonzero proper ideal of $R$ is invertible.

\begin{theorem}\label{T11}
\label{D}Let $R$ be a Dedekind domain and $I$ be a nonzero proper ideal of $R$. Then
$I$ is a 1-absorbing primary ideal of $R$ if and only if $\sqrt{I}$ is a
prime ideal of $R$.
\end{theorem}
\begin{proof}
If $I$ is a 1-absorbing primary ideal of $R$, then $\sqrt{I}$ is a prime ideal of $R$ by Theorem \ref{T-1}. Conversely, suppose$\sqrt{I}$ is a prime ideal of $R$. Since $R$ is a Dedekind domain, it is well-known that every nonzero prime ideal of $R$ is a maximal ideal of $R$. Thus $\sqrt{I}$
is a maximal ideal of $R$. Hence $I$ is a primary ideal of $R$, and thus $I$ is 1-absorbing primary ideal of $R$.
\end{proof}

\begin{theorem}\label{T12}
Let $R$ be a Noetherian integral domain that is
not a field. Then the following statements are equivalent:
\begin{enumerate}
\item $R$ is a Dedekind domain.
\item A nonzero proper ideal $I$ of $R$ is a 1-absorbing primary ideal of $R$
if and only if $I=P^{n}$ for some prime ideal $P$ of $R$ some positive
integer $n$.
\end{enumerate}
\end{theorem}
\begin{proof}
{\bf $(1)\Rightarrow (2)$}. Let $R$ be a Dedekind domain and $I$ be a nonzero
proper ideal of $R.$ Suppose that $I$ is a 1-absorbing primary ideal of $R$. Then $\sqrt{I} = P$ is a nonzero prime ideal of $R$ by Theorem \ref{T-1}. Since $R$ is a Dedekind domain, it is known that every nonzero prime ideal of $R$ is a maximal ideal of $R$. Thus $\sqrt{I} = P$ is a maximal ideal of $R$. Thus $I$ is a primary ideal of $R$. Since $R$ is Dedekind  and $I$ is a primary ideal of $R$ with $\sqrt{I} = P$ is a maximal ideal of $R$, we conclude that $I = P^n$ for some $n\geq 1$ by  of $R$ by \cite[Theorem 6.20]{LM}. Conversely, suppose that $I = P^n$ for some nonzero proper ideal $P$ of $R$ and a positive integer $n\geq 1$. Since $R$ is Dedekind, we conclude that $P$ is a maximal ideal of $R$, and hence $I$ is a primary ideal of $R$. Thus $I$ is a 1-absorbing primary ideal of $R$.

{\it $(2)\Rightarrow (1)$}. Suppose that every nonzero 1-absorbing primary ideals of $R$ is
of type $I = P^{n}$ for some nonzero prime ideal $P$ of $R$ and a positive integer $n \geq 1$.
Let $M$ be a maximal ideal of $R$. Since every ideal between $M^2$ and $M$ is a primary ideal and hence a 1-absorbing primary ideal of $R$, we conclude that there is no primary ideals of $R$ between $M^2$ and $M$. Thus $R$ is a Dedekind domain by \cite[Theorem 39.2]{Gilmer}.
\end{proof}

Since every principal ideal domain is a Dedekind domain, we have the following result.

\begin{corollary}\label{C1}
Let $R$ be a principal ideal domain and $I$ be a nonzero proper ideal of $R$. Then
$I$ is a 1-absorbing primary ideal of $R$ if and only if $I=p^{n}R$ for some
nonzero prime element $p$ of $R$ and a positive integer $n\geq 1$.
\end{corollary}

In light of Theorem \ref{T2} and Example \ref{e1}, observe that there are some rings whose 1-absorbing primary ideals are not of
the form $P^{n}$ for some prime ideal $P$ of $R$ and a positive integer $ n \geq 1$.

Let $I_{1}$ and $I_{2}$ be 1-absorbing primary ideals of $R$. If $\sqrt{I_{1}}\neq \sqrt{I_{2}}$, then $I_{1}\cap I_{2}$ needs not to be a 1-absorbing primary ideal
of $R$. We have the following example.
\begin{example}
Let $R = \mathbb{Z} \times \mathbb{Z}$. Then $I_1 = 4\mathbb{Z} \times \mathbb{Z}$ and $I_2 = \mathbb{Z} \times 9\mathbb{Z}$ are 1-absorbing primary ideals of $R$. Also,  $\sqrt{I_1} = 2\mathbb{Z} \times \mathbb{Z}$ and $\sqrt{I_2} = \mathbb{Z} \times 3\mathbb{Z}$. Hence $\sqrt{I_1} \not = \sqrt{I_2}$ and $I_1 \cap I_2 = 4\mathbb{Z} \times 9\mathbb{Z}$ is not a 1-absorbing primary ideal of $R$ by Theorem \ref{T1.5}
\end{example}

\begin{definition}
Let $I$ be a 1-absorbing primary ideal of a ring $R$. Then $\sqrt{I}=P$ is a prime ideal of $R$ by Theorem \ref{T-1}. Hence we call $I$
a $P$-1-absorbing primary ideal of $R$.
\end{definition}

So we have the following result.

\begin{theorem}\label{T13}
Let $I_{1},I_{2},...,I_{n}$ be $P$-1-absorbing primary ideals of a ring $R$. Then $
I = \bigcap_{i=1}^{n}I_{i}$ is a $P$-1-absorbing primary ideal of $R.$
\end{theorem}
\begin{proof}
First observe that $\sqrt{I} = P$. Suppose that $abc\in I$ for some nonunit elements $a,b,c\in R$ and $%
ab\notin I$. Without loss of generality, we may assume that $ab\notin
I_{1}.$ Since $I_{1}$ is a $P$-1-absorbing primary ideal of $R$ and $ab\notin I_1$, we have $c\in P$.
\end{proof}

\begin{theorem}\label{T14}
Let $R_{1}$ and $R_{2}$ be  rings and $
f:R_{1}\rightarrow R_{2}$ be a ring homomorphism such that $f(1) = 1$ and if $R_2$ is a quasilocal ring, then $f(a)$ is a nonunit of $R_2$ for every nonunit $a \in R_1$. Then the following statements hold.
\begin{enumerate}
\item Assume that $J$ is a 1-absorbing primary ideal of $R_{2}$, then $f^{-1}(J)$ is
a 1-absorbing primary ideal of $R_{1}$.
\item If $f$ is onto and $I$ is a 1-absorbing primary ideal of $R_{1}$ with $Ker(f) \subseteq I,$ then $f(I)$ is a 1-absorbing primary ideal of $R_{2}$.
\end{enumerate}
\end{theorem}
\begin{proof}
{\bf (1)}. Assume that $R_2$ is quasilocal. Let $abc\in f^{-1}(J)$ for some nonunit elements $a,b,c\in R_{1}$. Then $
f(abc)=f(a)f(b)f(c)\in J$ (note that $f(a), f(b)$, and $f(c)$ are nonunit elements of $R_2$ by hypothesis), which implies $f(a)f(b)\in J$ or $f(c)\in \sqrt{J}
$. It follows $ab\in f^{-1}(J)$ or $c\in \sqrt{f^{-1}(J)}=f^{-1}(\sqrt{J}).$
Thus $f^{-1}(J)$ is a 1-absorbing primary ideal of $R_{1}.$ Suppose that $R_2$ is not a quasilocal ring and $J$ is a 1-absorbing primary ideal of $R_{2}$. Then $J$ is a primary ideal of $R_2$ by Theorem 3. Thus $f^{-1}(J)$ is a primary ideal of $R_1$, and hence $f^{-1}(J)$ is
a 1-absorbing primary ideal of $R_{1}$.

{\bf (2)}. Since $f$ is onto and $Ker(f) \subseteq I$, we know that $f(\sqrt{I}) = \sqrt{f(I)}$. Let $xyz\in f(I)$ for some nonunit elements $x,y,z\in R_{2}$. Since $f$ is onto, there exist nonunit elements $a,b,c\in R_1$ such that $x = f(a),$ $y = f(b),$ and $ z = f(c)$. Hence $f(abc)=f(a)f(b)f(c)=xyz\in f(I)$. Since $Ker(f) \subseteq I$, we conclude that $abc\in I$. Hence $ab\in I$ or $c\in \sqrt{I}$; so $xy\in f(I)$ or $z\in f(\sqrt{I}) = \sqrt{f(I)}$. Thus $f(I)$ is a 1-absorbing primary ideal of $R_{2}$.
\end{proof}

In view of Theorem \ref{T14}, we have the following result.

\begin{corollary}\label{C2}
Let $I$ and $J$ be proper ideals of a ring $R$ with $I\subseteq J$ and suppose that if $R/I$ is a quasilocal ring, then $a + I$ is a nonunit of $R/I$ for every nonunit $a\in R$. Then $J$
is a 1-absorbing primary ideal of $R$ if and only if  $J/I$ is a 1-absorbing primary
ideal of $R/I$.
\end{corollary}
\begin{proof}
Let $f :  R \rightarrow R/I$ such that $f(a) = a + I$. Then $f$ is a ring homomorphism from $R$ onto $R/I$ and $f(1) = 1$. Suppose that $J$
is a 1-absorbing primary ideal of $R$. Since $Ker(f) = I \subseteq J$ and $f$ is onto, we conclude that $f(J) = J/I$ is a 1-absorbing primary ideal of $R/I$ by Theorem \ref{T14}(2). Suppose that $J/I$ is a 1-absorbing primary ideal of $R/I$. Then $f^{-1}(J/I) = J$ is a 1-absorbing primary ideal of $R$ by Theorem \ref{T14}(1).
\end{proof}

\begin{theorem}\label{T15}
Let $S$ be a multiplicatively closed subset of $R,$ and $I$ be a proper
ideal of $R$ \ Then the following statements hold.
\end{theorem}

\begin{enumerate}
\item If $I$ is a 1-absorbing primary ideal of $R$ such that $I\cap
S=\emptyset $, then $S^{-1}I\ $is a 1-absorbing primary ideal of $S^{-1}R.$

\item If $S^{-1}I$ is a 1-absorbing primary ideal of $S^{-1}R$ and $S\cap
Z_{I}(R)=\emptyset $, then $I\ $is a 1-absorbing primary ideal of $R.$
\end{enumerate}
\begin{proof}
(1) Let $\frac{a}{s_{1}}\frac{b}{s_{2}}\frac{c}{s_{3}}\in S^{-1}I$ for some
nonunit elements $a,b,c\in R,$ and $s_{1},s_{2},s_{3}\in S.$ Suppose that $\frac{a}{%
s_{1}}\frac{b}{s_{2}}\notin S^{-1}I$. Then $uabc\in I$ for some $u\in S$.
Since $I$ is 1-absorbing primary and $uab\notin I$, we have $c\in \sqrt{I}$.
Thus $\frac{c}{s_{3}}\in S^{-1}\sqrt{I}=\sqrt{S^{-1}I}$ which completes the
proof.

(2) Let $abc\in I$ for some nonunit elements $a,b,c\in R$. Hence $\frac{abc}{1}=%
\frac{a}{1}\frac{b}{1}\frac{c}{1}\in S^{-1}I$. Since $S^{-1}I$ is
1-absorbing primary, we have $\frac{a}{1}\frac{b}{1}\in S^{-1}I$ or $%
\frac{c}{1}\in \sqrt{S^{-1}I}=S^{-1}\sqrt{I}$. If $\frac{a}{1}\frac{b}{1}\in
S^{-1}I,$ then $uab\in I$ for some $u\in S.$ Since $u\notin Z_{I}(R)$, we
conclude that $ab\in I.$ If $\frac{c}{1}\in S^{-1}\sqrt{I},$ then $%
(tc)^{n}\in I$ for some positive integer $n$ and $t\in S$. Since $%
t^{n}\notin Z_{I}(R)$, we have $c^{n}\in I,$ as needed.
\end{proof}

\begin{theorem}\label{T16}
Let $I$ be a 1-absorbing primary ideal of a commutative ring $R.$ If $%
abJ\subseteq I$ for some nonunit elements $a,b\in R$ and a proper ideal $J$ of $R$%
, then $ab\in I$ or $J\subseteq \sqrt{I}.$
\end{theorem}
\begin{proof}
Assume on the contrary that $abJ\subseteq I$, but $ab\not\in I$ and $%
J\not\subseteq \sqrt{I}.$ Then there exists an element $j \in J$ such that $j\not\in
\sqrt{I}$. Hence we have $abj\in I$, but neither $ab\in I$ nor $j\in \sqrt{I}$, a contradiction.
\end{proof}

\begin{theorem}\label{T17}
\label{I}Let $I$ be a proper ideal of $R$. Then the followings statements are
equivalent.
\begin{enumerate}
\item $I$ is a 1-absorbing primary ideal of $R$.
\item For any proper ideals $I_{1},I_{2},I_{3}$ of $R$ such that $%
I_{1}I_{2}I_{3}\subseteq I$ implies that either $I_{1}I_{2}\subseteq I$ or $%
I_{3}\subseteq \sqrt{I}.$
\end{enumerate}
\end{theorem}
\begin{proof}
{\bf $(1)\Rightarrow (2)$}. Let $I$ be a 1-absorbing primary ideal of $R$. Assume
that $I_{1}I_{2}I_{3}\subseteq I$ for some proper ideals $%
I_{1},I_{2},I_{3}$ of $R$ and $I_1I_2 \nsubseteq I$. Then there exist nonunit elements $a\in I_{1}$
and $b\in I_{2}$ such that $ab \notin I$. Since $abI_3 \subseteq I$ and $ab \not\in I$, we conclude that $J \subseteq \sqrt{I}$ by Theorem \ref{T16}.

{\bf $(2)\Rightarrow (1)$}.  Suppose that $abc\in I$ for some nonunit elements $a,b,c\in R$ and $ab\notin I$.
Let $I_1 = aR, I_2 = bR$, and $I_3 = cR$. Then $I_1I_2I_3 \subseteq I$ and $I_1I_2\nsubseteq I$. Thus $I_3 = cR \subseteq \sqrt{I}$. Hence $c\in \sqrt{I}$.
\end{proof}

\end{document}